\documentclass{amsart}

\usepackage{amssymb}
\usepackage{amsmath}
\usepackage{url}
\usepackage{enumerate}
\usepackage{graphicx}
\usepackage[utf8]{inputenc}
\usepackage{datetime}
\usepackage{float}
\usepackage[normalem]{ulem}
\usepackage{mathtools}

\usepackage{hyperref}
\hypersetup{
    colorlinks  =   true,
    linkcolor   =   blue,
    filecolor   =   blue,
    urlcolor    =   blue,
    citecolor   =   blue,
}

\newcommand{\erase}[1]{}

\newtheorem{theorem}{Theorem}[section]
\newtheorem{lemma}[theorem]{Lemma}
\newtheorem{proposition}[theorem]{Proposition}
\newtheorem{corollary}[theorem]{Corollary}

\newtheorem{_algorithm}[theorem]{Algorithm}

\newtheorem{_definition}[theorem]{Definition}
\newenvironment{definition}{\begin{_definition}\rm}{\end{_definition}}

\newtheorem{_propositiondefinition}[theorem]{Proposition-Definition}

\newtheorem{_remark}[theorem]{\it Remark}
\newenvironment{remark}{\begin{_remark}\rm}{\end{_remark}}

\newtheorem{_example}[theorem]{Example}
\newenvironment{example}{\begin{_example}\rm}{\end{_example}}

\newtheorem{_assumption}[theorem]{Assumption}

\newtheorem{_construction}[theorem]{Construction}

\newtheorem{_claim}[theorem]{Claim}

\newtheorem{_conjecture}[theorem]{Conjecture}

\newtheorem{_problem}[theorem]{Problem}

\numberwithin{equation}{section}
\numberwithin{table}{section}
\numberwithin{figure}{section}
\renewcommand{\qed}{\hfill {$\Box$}}

\renewcommand{\P}{\mathord{\mathbb  P}}
\newcommand{\Q}{\mathord{\mathbb  Q}}
\newcommand{\R}{\mathord{\mathbb R}}

\newcommand{\Z}{\mathord{\mathbb Z}}

\newcommand{\CCC}{\mathord{\mathcal C}}

\newcommand{\LLL}{\mathord{\mathcal L}}

\newcommand{\NNN}{\mathord{\mathcal N}}

\newcommand{\PPP}{\mathord{\mathcal P}}

\newcommand{\RRR}{\mathord{\mathcal R}}

\newcommand{\mapdownsurj}{
\hbox{$\bigm\downarrow$}
\llap{\hbox{\raise 2pt\hbox{$\bigm\downarrow$}}}%
\vstrechmapdown
}

\newcommand{\mapupsurj}{
\hbox{$\bigm\uparrow$}
\llap{\hbox{\raise 2pt\hbox{$\bigm\uparrow$}}}%
\vstrechmapup
}

\newcommand{\inj}{\hookrightarrow}

\newcommand{\isom}{\xrightarrow{\sim}}

\newcommand{\set}[2]{\{\,{#1}\mid {#2} \,\}}

\newcommand{\gen}[1]{\langle {#1}  \rangle}

\newcommand{\ol}{\overline}

\newcommand{\tensor}{\otimes}

\newcommand{\sprime}{\sp\prime}

\newcommand{\spcirc}{\sp{\mathord{\circ}}}

\newcommand{\sperp}{\sp{\perp}}

\newcommand{\dual}{\sp{\vee}}

\newcommand{\semidirectproduct}{\rtimes}

\newcommand{\inv}{\sp{-1}}

\newcommand{\bdr}{\partial\,}

\newcommand{\OG}{\mathord{\mathrm{O}}}

\newcommand{\rank}{\operatorname{\mathrm{rank}}\nolimits}

\newcommand{\mystruth}[1]{\phantom{\hbox{\vrule height #1}}}
\newcommand{\mystrutd}[1]{\phantom{\hbox{\vrule depth #1}}}

\newcommand{\intf}[1]{\langle #1 \rangle}

\newcommand{\Rats}{\mathord{\mathrm{Rats}}}

\newcommand{\barPPP}{\overline{\PPP}}
\newcommand{\barCCC}{\overline{\CCC}}

\newcommand{\barNNN}{\overline{\NNN}}
\newcommand{\MW}{\mathrm{MW}}
\newcommand{\ADE}{\mathrm{ADE}}

\newcommand{\tilTheta}{\widetilde{\Theta}}

\newcommand{\tilSigma}{\widetilde{\Sigma}}

\newcommand{\intfL}[1]{\intf{#1}_L}

\newcommand{\zerosec}{\zeta}

\newcommand{\Leech}{\Lambda}

\newcommand{\Card}{\mathrm{Card}}

\newcommand{\XX}{\mathbb{X}}

\newcommand{\spm}{^{-}}
\newcommand{\intfNm}[1]{\intf{#1}_N\spm}
\newcommand{\Nm}{N\spm}

\newcommand{\LatII}{\mathord{I\hskip -.7pt I}}
\newcommand{\nullvars}{\phantom{i},\phantom{i}}

\newcommand{\Coinfty}{\mathord{\mathrm{Co}}_{\infty}}

\newcommand{\XXi}{\,\Xi}

\newcommand{\Lts}{L_{26}}

\newcommand{\tauN}{\tau(N)}

\newcommand{\code}{\varGamma}

\newcommand{\newLeech}{\varLambda}

\begin{document}

\title[Construction of the Leech lattice]
{A note on construction of the Leech lattice}

\author[Ichiro  Shimada]{Ichiro Shimada}
\address{Department of Mathematics, Graduate School of Science, Hiroshima University,
1-3-1 Kagamiyama, Higashi-Hiroshima, 739-8526 JAPAN}
\email{ichiro-shimada@hiroshima-u.ac.jp}

\begin{abstract}
In this paper, we present a method to construct
the Leech lattice from other Niemeier lattices.
\end{abstract}
\keywords{Leech lattice, Niemeier lattice, deep hole}
%\date{  \currenttime \;\;\today}

\subjclass[2020]{11H06}
\thanks{This work was supported by JSPS KAKENHI Grant Number~20H01798.}
\maketitle
\section{Introduction}\label{sec:introduction}
The Leech lattice 
is an important mathematical object.
Various constructions of this beautiful lattice are known.
In this paper, we give a method to construct
the Leech lattice as modifications of other Niemeier lattices.
\par
Let $N$ be a Niemeier lattice with the intersection form $\intf{\phantom{a}, \phantom{a}}_N$ . %,
Suppose that $N$ is not isomorphic to the Leech lattice. 
Let $R$ be the set of vectors in $N$ of square-norm $2$, 
and $\gen{R}\subset N$  the sublattice generated by $R$.
Then $N/\gen{R}$ is a finite abelian group, which we call the \emph{code} of $N$.
(In~\cite{CSEnumeration1982Note}, it is called the \emph{glue code}.)
For each codeword of this code,
we give a method to construct the Leech lattice.
The main result is given in Theorem~\ref{thm:constC}.
\par
Considering the case where this codeword is $0$,
we obtain the following.
Let $N$ and $R$ be as above.
We choose a simple root system $\Theta$ of $R$,
that is, $\Theta$ is a set of vectors in $R$ that form a basis of $\gen{R}$ 
with the dual graph being a Dynkin diagram of an ordinary $\ADE$-type.
Then we have a vector $\rho\in  N\tensor\Q$ such that $\intf{r, \rho}_N=-1$ for all  $r\in \Theta$.
In fact, we have $\rho\in N$.
Let $h$ be the Coxeter number of $N$.
\begin{corollary}\label{cor:zerogamma} 
The $\Z$-module 
\[
\set{u\in N}{\intf{u, \rho}_N \equiv 0 \bmod 2h+1}
\]
 with the quadratic form
\[
u\mapsto \intf{u, u}_N-\frac{2}{(2h+1)^2}\,(\intf{u, \rho}_N)^2
\]
is isomorphic to the Leech lattice.
\qed
\end{corollary}
Our construction is similar to the twenty-three holy constructions 
%due to 
of Conway and Sloane~\cite{CS1982Note}. 
See also~\cite{CSLoretzian1982Note}.
In the proof of our construction, 
the classification of deep holes of the Leech lattice~\cite{CPS1982Note}
and 
the determination of the fundamental domain of the Weyl group 
of the even unimodular Lorentzian lattice $\LatII_{25, 1}$~\cite{Conway1983Note}
play important roles.
Some ideas 
in our construction have already appeared in Borcherds~\cite{Borcherds1985}.
\par
A novelty of our approach is that we use the geometry of $K3$ surfaces 
as a heuristic guide.
Let $\XX$ be a $K3$ surface whose N\'eron-Severi lattice $S_{\XX}$
is an even unimodular lattice $\LatII_{1,25}$ of signature $(1,25)$.
It is needless to say that such a $K3$ surface  $\XX$ does \emph{not} exist.
Applying the lattice theoretic tools for the study of $K3$ surfaces to this \emph{virtual} $K3$ surface $\XX$,
however, 
we can rephrase the above mentioned results of Conway et al.~\cite{Conway1983Note, CPS1982Note}
in terms of geometry of $K3$ surfaces.
Then, using an algorithm
that we have developed in~\cite{Shimada2022} for the study of $K3$ surfaces, 
we obtain our construction.
\par
Note that, in the proof of our main result (Theorem~\ref{thm:constC}), we do not use $\XX$,
and hence 
the result is rigorously correct.
In fact, we can verify our result by direct computations.
See Remark~\ref{rem:canonicalreps}.
Computational data relevant to this paper is presented in~\cite{compdataLN}
in the format of {\tt GAP}~\cite{GAP}. 
\par \medskip
{\bf Notation.}
A lattice is a free $\Z$-module of finite rank
with a non-degenerate symmetric bilinear form that takes values in $\Z$.
This bilinear form is called the \emph{intersection form} of the lattice.
For a lattice $M$ with the intersection form $\intf{\nullvars}_M$,
let $M\spm$ denote the lattice whose underlying $\Z$-module is $M$
and whose intersection form is equal to $\intf{\nullvars}\spm_M:=-\intf{\nullvars}_M$.
\par
Let $\Leech$ be the Leech lattice.
In this paper, we construct 
the \emph{negative-definite} Leech lattice $\Leech\spm$ from 
\emph{negative-definite} Niemeier lattices $N\spm$ using 
the result on the lattice $\LatII_{25, 1}\spm=\LatII_{1,25}$.
We make this change of sign 
because we want to use ideas coming from the geometry of algebraic surfaces.
\section{Preliminaries}\label{sec:Preliminaries}
\subsection{Roots and reflections}\label{subsec:roots}
A lattice $M$ is \emph{even} if $\intf{v,v}_M\in 2\Z$ holds for all $v\in M$.
Let $M$ be an even lattice.
A vector $r$ of $M$ is said to be a \emph{root} if $|\intf{r,r}_M|=2$.
A lattice is said to be a \emph{root lattice} if it is generated by roots.
A root $r$ of $M$ with $\intf{r, r}_M=\pm 2$ defines the \emph{reflection}
\[
s_r\colon x\mapsto x\mp \intf{x, r}_M\cdot  r,
\]
which is an element of the automorphism group $\OG(M)$ of the lattice $M$.
Let $M\dual$ denote the \emph{dual lattice} 
\[
\set{x\in M\tensor\Q}{\intf{x, v}_M\in \Z \;\;\textrm{for all}\;\; v\in M}.
\]
The \emph{discriminant group} of $M$ is 
the finite abelian group $M\dual/M$.
The group $\OG(M)$ acts on $M\dual$, and hence on $M\dual/M$.
Since $s_r(x)-x\in M$ holds for any $x\in M\dual$,
the reflections $s_r\in \OG(M)$ act on $M\dual/M$ trivially.
\subsection{ $\ADE$-configurations}\label{subsec:Theta}
A \emph{$(-2)$-vector} of an even lattice $M$ is a vector $r\in M$ such that $\intf{r,r}_M=-2$.
Let $\{r_1, \dots, r_m\}$ be a set of $(-2)$-vectors of an even lattice $M$
such that $\intf{r_i, r_j}_M \in \{0, 1, 2\}$ holds for any $i, j$ with $i\ne j$.
The \emph{dual graph} of $\{r_1, \dots, r_m\}$ is the graph whose set of nodes is $\{r_1, \dots, r_m\}$
and whose set of simple edges (resp.~of double edges) is the set of pairs $\{r_i, r_j\}$ such that $\intf{r_i, r_j}_M=1$ (resp.~$\intf{r_i, r_j}_M=2$).
(A double edge appears only in extended $\ADE$-configuration whose $\ADE$-type contains $A_1$.)
We say that $\{r_1, \dots, r_m\}$ forms an \emph{ordinary $\ADE$-configuration} (resp.~an \emph{extended $\ADE$-configuration})
if the dual graph is the ordinary Dynkin diagram (resp.~the extended Dynkin diagram) of type $\ADE$,
and in this case, we define 
the type of the configuration 
to be the type of the dual graph.
 \par
 Let $M$ be a negative-definite root lattice, and 
$R$ the set of $(-2)$-vectors in $M$.
 We have $M=\gen{R}$.
A \emph{simple root system} of  $M$ is a subset of $R$ that is a basis of  $M$ and 
that forms an ordinary $\ADE$-configuration.
A negative-definite root lattice always has a simple root system.
The set of simple root systems of $M$ is described as follows.
 We denote by $W(M)$  
the subgroup of $\OG(M)$
generated by all the reflections 
$s_r$ with respect to $r\in R$,
and call it the \emph{Weyl group} of $M$.
We consider the unit sphere
\[
S:=\set{x\in M\tensor\R}{\intf{x,x}_M=-1}.
\]
For $r\in R$, we put
\[
(r)\sperp:=\set{x\in S}{\intf{x, r}_M=0},
\quad
H^{+}(r):=\set{x\in S}{\intf{x, r}_M>0},
\]
and set 
\[
S\spcirc:=S\; \setminus\;\bigcup_{r\in R}\; (r)\sperp.
\]
Then $W(M)$ acts on the set of connected components of $S\spcirc$ 
simple-transitively.
A subset $\Theta$ of $R$ is a simple root system of $M$
if and only if the space 
\[
\Delta(\Theta):=\bigcap_{r\in \Theta}\; H^{+}(r)
\]
is a connected component of $S\spcirc$ and, 
for each $r\in \Theta$,
the intersection 
 $(r)\sperp\cap \ol{\Delta(\Theta)}$ contains 
a non-empty open subset of $(r)\sperp$, where $\ol{\Delta(\Theta)}$
is the closure of $\Delta(\Theta)$ in $S$.
Therefore 
$W(M)$ acts on the set of simple root systems of $M$
simple-transitively.
\subsection{The coefficients of the highest root}\label{subsec:m}
Let $\tilSigma$ be a set of $(-2)$-vectors 
that form an extended $\ADE$-configuration of type $\tau$
with the dual graph being \emph{connected},
that is, we have $\tau=A_l$ or $\tau=D_m$ or $\tau=E_n$.
We choose a vector $\theta\in \Sigma$ such that $\Sigma:=\tilSigma\setminus \{\theta\}$ 
forms an ordinary $\ADE$-configuration of type $\tau$.
Then $\Sigma$ is a simple root system
of the negative-definite root lattice $\gen{\Sigma}$.
Let  $\mu\in \gen{\Sigma}$ be the highest root with respect to $\Sigma$ (see~\cite[Section 1.5]{Ebeling2013} for the definition).
We define a function $m\colon \tilSigma \to \Z_{>0}$ by
\[
m(r):=\begin{cases}
1 & \textrm{if $r=\theta$}, \\
\textrm{the coefficient of $r$ in $\mu$} &\textrm{if $r\in \Sigma$}.
\end{cases}
\]
In fact, 
the function $m$ 
does not depend on the choice of $\theta \in \tilSigma$.
We have $m(r)=1$ if and only if $\tilSigma\setminus \{r\}$ forms an ordinary $\ADE$-configuration of type $\tau$.
The values of the function $m$ for each connected $\ADE$-type $\tau$ 
can be found in~\cite[Figure 23.1]{CSbook} or~in~\cite[Figure 1.8]{Ebeling2013}. 
\par
Let $\Z^{\tilSigma}$ be the $\Z$-module of functions $\tilSigma\to \Z$.
We can also define $m$ as 
the function $\tilSigma\to \Z$ that takes values in $\Z_{>0}$ and that is 
the generator of the following $\Z$-submodule of of $\Z^{\tilSigma}$:
\[
\set{(\;x(r)\;)\in \Z^{\tilSigma}}{\textstyle{\sum_{r\in \tilSigma} }\;x(r) \intf{r, r\sprime}=0 \;\;\textrm{for all $r\sprime\in \tilSigma$}}.
\]
\subsection{Hyperbolic lattices}\label{subsec:hyperboloc}
We say that
a lattice $M$ of rank $n>1$ is \emph{hyperbolic} if its signature is  $(1, n-1)$.
\begin{example}\label{example:Ln}
If $n$ is a positive integer satisfying  $n\equiv 2 \bmod 8$,
then there exists an even unimodular hyperbolic lattice $L_n$ of rank $n$, and $L_n$ is unique up to isomorphism.
The lattice $L_{26}\cong \LatII_{1,25}$ plays a central role in this paper.
\end{example}
\begin{example}\label{example:U}
The \emph{hyperbolic plane}  $U$  is 
the even unimodular hyperbolic lattice $L_2$ of rank $2$.
We fix a basis $u_0, u_1$ of $U$ such that the Gram matrix of $U$ with respect to $u_0, u_1$ is 
\[\left[\begin{array}{cc} 
0 & 1 \\ 
1 & 0 
\end{array}\right].
\]
\end{example}
Let $M$ be an even hyperbolic lattice.
A \emph{positive cone} of $M$ is 
one of the two connected components of 
the space $\set{x\in M\tensor \R}{\intf{x, x}_M>0}$.
Let $\PPP_M$ be a positive cone of $M$.
We denote by $\barPPP_M$ the closure of $\PPP_M$ in $M\tensor\R$,
and by $\bdr\barPPP_M$ the boundary $\barPPP_M\setminus \PPP_M$.
We put
\[
\RRR_M:=\set{r\in M}{\intf{r,r}_M=-2}, 
\]
and, for $r\in \RRR_M$, let $(r)\sperp$ be the hyperplane 
$\set{x\in \PPP_M}{\intf{x, r}_M=0}$ in $\PPP_M$.
The \emph{Weyl group} $W(M)$ 
is a subgroup of $\OG(M)$ generated by all the reflections 
$s_r$
associated with $r\in \RRR_M$.
Then $W(M)$ acts on $\PPP_M$.
\begin{definition}
A \emph{standard fundamental domain of $W(M)$} is the closure in $\PPP_M$ of a connected component 
of 
\[
\PPP_M\;\setminus\; \bigcup_{r\in  \RRR_M} (r)\sperp. 
\]
Let $D$ be a standard fundamental domain of $W(M)$.
We say that $r\in \RRR_M$ \emph{defines a wall of $D$}
 if $D\cap (r)\sperp$ contains a non-empty open subset of 
 $(r)\sperp$ and $\intf{x, r}_M\ge 0$ holds for all $x\in D$.
 \end{definition}
Note that, contrary to the case where $M$ is a negative-definite root lattice,
a standard fundamental domain of $W(M)$ in the positive cone of hyperbolic lattice may have infinitely many walls.
\section{Niemeier lattices}\label{sec:Niemeier}
A \emph{Niemeier lattice} is an even positive-definite unimodular lattice of rank $24$.
Niemeier~\cite{Niemeier} classified 
Niemeier lattices.
It was shown in~\cite{Niemeier} that, up to isomorphism, 
there exist exactly $24$ Niemeier lattices, 
that one of them contains no roots,
whereas each of the other $23$ lattices contains a sublattice of finite index generated by roots.
See also~\cite{CSEnumeration1982Note}.
The Niemeier lattice containing no roots is called the \emph{Leech lattice},
and is denoted by $\Leech$.
\par
Let $N$ be a Niemeier lattice with roots.
Let $R$ be the set of $(-2)$-vectors of $\Nm$,
and $\gen{R}$  the sublattice of $\Nm$ generated by $R$.
Then the \emph{code}  $\Nm/\gen{R}$ of $N\spm$ is a subgroup of the discriminant group $\gen{R}\dual/\gen{R}$.
Let $\Theta\subset R$ be a simple root system of $\gen{R}$, 
and 
let $\tauN$ denote the $\ADE$-type of   $\Theta$.
We call $\tauN$ the \emph{$\ADE$-type} of $N\spm$.
Note that  $\OG(\Nm)$ is a subgroup of $\OG(\gen{R})$.
More precisely, we have 
\[
\OG(\Nm)=\set{g \in \OG(\gen{R})}{\textrm{the action of $g$ on $\gen{R}\dual/\gen{R}$ preserves $\Nm/\gen{R}$}}.
\]
Since the subgroup $W(\gen{R})$  of $\OG(\gen{R})$ 
acts on $\gen{R}\dual/\gen{R}$ trivially,
we see that  $W(\gen{R})$ is contained in $\OG(\Nm)$.
Since $W(\gen{R})$ acts on the set of simple root systems of $\gen{R}$
transitively, we have the following:
\begin{proposition}\label{prop:transitive}
The group $\OG(\Nm)$ acts on the set of simple root systems of $\gen{R}$  transitively.
\qed
\end{proposition}
We write 
\begin{equation}\label{eq:taudecomp}
\tauN=\tauN_1+\cdots +\tauN_K,
\end{equation}
where  $\tauN_1, \dots, \tauN_K$ are the $\ADE$-types of the connected components of  the dual graph of $\Theta$.
Accordingly, we obtain the decompositions
\begin{equation}\label{eq:ThetaRRdecomp}
\Theta=\Theta_1\sqcup\dots\sqcup \Theta_K,
\quad
R=R_1\sqcup\dots\sqcup R_K,
\end{equation}
in such a way that  
$\Theta_i\subset R_i$ is a simple root system 
of the root lattice $\gen{R_i}$ of type $\tauN_i$.
We put
\[
n_i:=\Card(\Theta_i)=\rank{\gen{R_i}}.
\]
Then we have $24=n_1+\dots+n_K$.
\begin{definition}\label{def:Coxeter number}
The \emph{Coxeter number} $h_i$ of $R_i$ 
is defined by any of  the following:
\begin{enumerate}[(a)]
\item \label{h:a} $\Card(R_i) =n_i h_i$.
\item  \label{h:b}  Let $\rho_i\in \gen{R_i}\tensor\Q$ be the vector satisfying $\intfNm{\rho_i, r}=1$ for all $r\in \Theta_i$.
Then we have $\intfNm{\rho_i, \rho_i}=-n_i h_i (h_i+1)/12$.
\item  \label{h:c}  
Let $\mu_i\in R_i$ denote the highest root with respect to $\Theta_i$.
(See Section 1.5 of~\cite{Ebeling2013}).
Then we have $h_i=\intfNm{\mu_i, \rho_i}+1$, that is,
$h_i-1$ is the sum of the coefficients $m(r)$ of $\mu_i$ expressed as a linear combination of vectors $r\in \Theta_i$,
where $m$ is the function defined in Section~\ref{subsec:m}.
\item  \label{h:d}  The product of all the reflections $s_r$ with respect to $r\in \Theta_i$ (where the product is taken in arbitrary order)
is of order $h_i$ in $\OG(\gen{R_i})$.
\end{enumerate}
\end{definition}
A remarkable fact  about the  $\ADE$-type $\tauN$  is that 
 $h_i$ does not depend on $i$. 
 We put
\[
h:=h_1=\dots=h_K,
\]
and call it the \emph{Coxeter number} of $\Nm$.
(See Table~\ref{table:N}.)
We also put
\[
\rho:=\rho_1+\dots+\rho_K, 
\]
which is called a \emph{Weyl vector} of $N\spm$.
By property (b) of $h$ above,  we have 
\begin{equation}\label{eq:rhonorm}
\intfNm{\rho, \rho}=-2h(h+1).
\end{equation}
\begin{remark}\label{rem:rho}
The Weyl vector  $\rho \in \gen{R}\tensor\Q=\Nm\tensor\Q$ is 
in fact
a vector of  $\Nm$.
This fact can be easily confirmed by direct computation.
Borcherds~\cite{Borcherds1985} gave a proof of this fact.
See also Remark~\ref{rem:rho2}.
\end{remark}
\begin{table}
{\small 
\[
% Ntable.tex
\begin{array}{lllll}
\textrm{No.} \phantom{aa} & \tau \phantom{aaaa} & h \phantom{aaaa}& N\spm /\gen{R} \phantom{aaaa}& \gen{R}\dual /\gen{R}\mystrutd{7pt}\\ 
 \hline 
1 & D_{24} & 46 & \Z/2\Z & (\Z/2\Z)^{2} \mystruth{11pt} \\ 
2 & 3E_{8} & 30 & 0 & 0 \\ 
3 & D_{16}+E_{8} & 30 & \Z/2\Z & (\Z/2\Z)^{2} \\ 
4 & A_{24} & 25 & \Z/5\Z & \Z/25\Z \\ 
5 & 2D_{12} & 22 & (\Z/2\Z)^{2} & (\Z/2\Z)^{4} \\ 
6 & A_{17}+E_{7} & 18 & \Z/6\Z & \Z/2\Z\times \Z/18\Z \\ 
7 & D_{10}+2E_{7} & 18 & (\Z/2\Z)^{2} & (\Z/2\Z)^{4} \\ 
8 & A_{15}+D_{9} & 16 & \Z/8\Z & \Z/4\Z\times \Z/16\Z \\ 
9 & 3D_{8} & 14 & (\Z/2\Z)^{3} & (\Z/2\Z)^{6} \\ 
10 & 2A_{12} & 13 & \Z/13\Z & (\Z/13\Z)^{2} \\ 
11 & 4E_{6} & 12 & (\Z/3\Z)^{2} & (\Z/3\Z)^{4} \\ 
12 & A_{11}+D_{7}+E_{6} & 12 & \Z/12\Z & \Z/3\Z\times \Z/4\Z\times \Z/12\Z \\ 
13 & 4D_{6} & 10 & (\Z/2\Z)^{4} & (\Z/2\Z)^{8} \\ 
14 & 2A_{9}+D_{6} & 10 & \Z/2\Z\times \Z/10\Z & (\Z/2\Z)^{2}\times (\Z/10\Z)^{2} \\ 
15 & 3A_{8} & 9 & \Z/3\Z\times \Z/9\Z & (\Z/9\Z)^{3} \\ 
16 & 2A_{7}+2D_{5} & 8 & \Z/4\Z\times \Z/8\Z & (\Z/4\Z)^{2}\times (\Z/8\Z)^{2} \\ 
17 & 4A_{6} & 7 & (\Z/7\Z)^{2} & (\Z/7\Z)^{4} \\ 
18 & 6D_{4} & 6 & (\Z/2\Z)^{6} & (\Z/2\Z)^{12} \\ 
19 & 4A_{5}+D_{4} & 6 & \Z/2\Z\times (\Z/6\Z)^{2} & (\Z/2\Z)^{2}\times (\Z/6\Z)^{4} \\ 
20 & 6A_{4} & 5 & (\Z/5\Z)^{3} & (\Z/5\Z)^{6} \\ 
21 & 8A_{3} & 4 & (\Z/4\Z)^{4} & (\Z/4\Z)^{8} \\ 
22 & 12A_{2} & 3 & (\Z/3\Z)^{6} & (\Z/3\Z)^{12} \\ 
23 & 24A_{1} & 2 & (\Z/2\Z)^{12} & (\Z/2\Z)^{24}  
\end{array}
% End of Ntable.tex

\]
}
\vskip .2cm
\caption{Niemeier lattices with roots}\label{table:N}
\end{table}
%
%%%%%%%%%%%%%%%%%%%%%%%%%%%%%%
%
\section{Deep holes}\label{sec:Deepholes}
In this section, 
we review the classification of deep holes of the Leech lattice $\Leech$  
due to Conway, Parker, Sloane~\cite{CPS1982Note},
and the determination of the fundamental domain of the Weyl group $W(\Lts)$
%of the even unimodular hyperbolic lattice 
of $\Lts\cong \LatII_{1,25}$
due to Conway~\cite{Conway1983Note}.
\par
For $x, y\in \Leech\tensor\R$, we put
\[
d(x, y):=\sqrt{\intf{x-y, x-y}_{\Leech}}\;, \quad \textrm{and}\quad 
d(x, \Leech):=\min{}_{\lambda\in \Leech} \,d(x, \lambda).
\]
The \emph{covering radius} of $\Leech$ is defined to be 
the maximum of $d(x, \Leech)$,
where $x$ runs through $\Leech\tensor\R$.
In~\cite{CPS1982Note}, the following was proved:
\begin{theorem}\label{thm:coveringradius}
The covering radius of $\Leech$ is $\sqrt{2}$.
\qed
\end{theorem} 
Using Vinberg's algorithm~\cite{Vinberg1975}
and Theorem~\ref{thm:coveringradius}, 
Conway~\cite{Conway1983Note} proved the following.
Let  $\intfL{\phantom{a},\phantom{a}}$ denote  the intersection form of $\Lts\cong \LatII_{1,25}$.
Let $\RRR_L$ be the set of $(-2)$-vectors of $\Lts$.
We choose a positive cone $\PPP_L$ of $\Lts$.
\begin{definition}
We call 
a standard fundamental domain of $W(\Lts)$ in $\PPP_L$ 
a \emph{Conway chamber}.
A non-zero primitive vector $w\in \Lts$ is called a \emph{Weyl vector}
if $w\in  \bdr\barPPP_L$ (in particular, 
we have $\intf{w, w}_L=0$) and the lattice $(\Z w)\sperp/\Z w$ is isomorphic to $\Leech\spm$.
For a Weyl vector $w$, we put
\begin{eqnarray*}
\LLL(w) &:=& \set{r\in \RRR_L}{\intfL{w, r}=1},\\
\CCC(w) &:=& \set{x\in \PPP_L}{\intfL{x, r}\ge 0\;\;\textrm{for all}\;\; r\in \LLL(w)}.
\end{eqnarray*}
An element of $\LLL(w)$ is called a \emph{Leech root} of the Weyl vector $w$.
\end{definition}
\begin{theorem}[Conway~\cite{Conway1983Note}]\label{thm:ConwayChamber}
\begin{enumerate}[{\rm (1)}]
\item
The mapping $w\mapsto\CCC(w)$
gives a bijection between the set of Weyl vectors and the set of Conway chambers.
\item
Let $w$ be a Weyl vector.
A $(-2)$-vector $r$ of $\Lts$ defines a wall of the Conway chamber $\CCC(w)$ if and only if $r\in \LLL(w)$.
\qed
\end{enumerate}
\end{theorem}
Let  $\Coinfty$ denote the group of affine isometries of the Leech lattice  $\Leech$.
We have $\Coinfty=\Leech\semidirectproduct \OG(\Leech)$,
where $\Leech$ acts on $\Leech$ by translation.
Let $U_{\Leech}$ be a copy of the hyperbolic plane $U$, and we put
\begin{equation*}\label{eq:LULeech}
L_{\Leech}:=U_{\Leech}\oplus \Leech\spm, 
\end{equation*}
which is isomorphic to $L_{26}$.
We write elements of $L_{\Leech}$ as 
\begin{equation*}\label{eq:abv}
(a, b, v)_{\Leech}:=a u_0 + b u_1 +v,
\quad\textrm{where $a, b\in \Z$ and $v\in \Leech\spm$},
\end{equation*}
where $u_0$ and $u_1$ are the basis of $U$ given in Example~\ref{example:U}.
Then the vector 
\[
w_{\Leech}:=(1,0,0)_{\Leech}
\]
is a Weyl vector of $L_{\Leech}$, and the mapping 
\begin{equation}\label{eq:Lbij}
\lambda\;\;\mapsto\;\; r_{\lambda}:=(\; -1-\lambda^2/2, \; 1, \; \lambda\;)_{\Leech},
\quad\textrm{where $\lambda^2=\intf{\lambda, \lambda}\spm_{\Leech}$, }
\end{equation}
gives a bijection $\Leech\spm\cong\LLL(w_{\Leech})$.
Then Theorem~\ref{thm:ConwayChamber}
implies the following:
\begin{corollary}\label{cor:Coinfty}
The automorphism group
\[
\OG(L_{\Leech}, w_{\Leech}):=\set{g\in \OG(L_{\Leech})}{w_{\Leech}^g=w_{\Leech}}
\]
of the Conway chamber $\CCC(w_{\Leech})$ is isomorphic to  $\Coinfty$ via  the bijection 
$r_{\lambda} \mapsto \lambda$.
\qed
\end{corollary}
A point $c\in \Leech\spm\tensor\R$ is called a \emph{deep hole}
if $c$ satisfies   $d(c, \Leech)= \sqrt{2}$.
The group $\Coinfty$ acts on the set of deep holes.
In~\cite{CPS1982Note}, deep holes are classified up to the action of $\Coinfty$.
For a deep hole $c\in \Leech\tensor\R$,
we put
\[
P_0(c):=\set{\lambda\in \Leech\spm}{d(c, \lambda)=\sqrt{2}},
\]
and call it the set of \emph{vertices} of $c$.
We then consider the set 
\begin{equation*}\label{eq:PLc}
\XXi_0 (c):=\set{r_\lambda\in \LLL(w_{\Leech}) }{\lambda\in P_0(c)}.
\end{equation*}
\begin{theorem}[Conway, Parker, Sloane~\cite{CPS1982Note}]
\begin{enumerate}[{\rm (1)}]
\item 
 For each deep hole $c$,  the set $\XXi_0 (c)$  forms 
an extended   $\ADE$-configuration,
and its type $\tau(c)$   
is one of the $23$ types $\tauN$ of Niemeier lattices  $N$ with roots.
\item Conversely, 
for an  $\ADE$-type $\tauN$ of a Niemeier lattice  $N$ with roots,
there exists a deep hole $c$ such that $\tau(c)=\tauN$.
\item 
 Two deep holes $c$ and $c\sprime$ are $\Coinfty$-equivalent if and only if
$\tau(c)=\tau(c\sprime)$.
\qed
\end{enumerate}
\end{theorem}
\begin{definition}
The \emph{Coxeter number of a deep hole $c$} is the Coxeter number of 
the Niemeier lattice  $N$ with roots such that $\tauN=\tau(c)$.
\end{definition}
%
%It turns out that the Coxeter number $h$ of a deep hole $c$ is the least positive integer such that $hc\in \Leech$.
%
%%%%%%%%%%%%%%%%%%%%%%%%%%%%%%
%
\section{The shape of a Conway chamber}\label{sec:shape}
For each $\Coinfty$-equivalence class of deep holes,
we have computed a representative element $c$ and its vertices $P_0(c)$
explicitly in~\cite{Shimada2017}.
The results are also presented  on the web page~\cite{compdataLN}.
Using this data, we can confirm the following:
\begin{proposition}\label{prop:primitive}
Let  $c\in \Leech\spm\tensor\Q$ be a deep hole, and 
$h$ the Coxeter number of $c$.
Then  $hc$ is a primitive vector of $\Leech\spm$, and $h\intf{c, c}_{\Leech}\spm /2$ is an integer.
\qed
\end{proposition}
Recall that $L_{\Leech}=U_{\Leech}\oplus \Leech\spm$.
We consider the Conway chamber $\CCC(w_{\Leech})$
of $L_{\Leech}$ corresponding to the Weyl vector $w_{\Leech}=(1,0,0)_{\Leech}$.
Let $\overline{\CCC}(w_{\Leech})$ be the closure of  $\CCC(w_{\Leech})$
in $\barPPP_L$.
%\par
Let $c\in \Leech\spm\tensor\Q$ be a deep hole with the Coxeter number $h$.
We put 
\begin{equation}\label{eq:definitionfc}
\bar{f}(c):=(\; -\intf{c,c}\spm_{\Leech}/2, \; 1, \;c\;)_{\Leech}\;\in\;L_{\Leech}\tensor \Q,
\quad
f(c):= h \bar{f}(c).
\end{equation}
By Proposition~\ref{prop:primitive}, 
 the vector $f(c)$ is a primitive vector of $L_{\Leech}$ with $\intf{f(c), f(c)}_L=0$.
Then we have the following.
\begin{proposition}\label{prop:ray}
The intersection  $\barCCC(w_{\Leech})\cap \bdr \barPPP_L$ is a union of the half-lines 
$\R_{\ge 0} w_{\Leech}$ and $\R_{\ge 0} f(c)$,
where $c$ runs through the set of deep holes.
\end{proposition}
\begin{proof}
It is obvious that $\R_{\ge 0} w_{\Leech}\subset \barCCC(w_{\Leech})\cap \bdr \barPPP_L$.
Let $\ell$ be a  point of $ \bdr \barPPP_L \;\setminus \; \R_{\ge 0} w_{\Leech}$.
Then we have $\intf{\ell, w_{\Leech}}_L>0$.
Rescaling $\ell$ by a positive real number,
we assume that $\intf{\ell, w_{\Leech}}_L=1$
so that
we have 
\[
\ell=(\; -\intf{v,v}\spm_{\Leech}/2, \; 1, \;v\;)_{\Leech}
\]
for some $v\in \Leech \tensor\R$.
Then, for each $\lambda\in \Leech\spm$,  we have 
\begin{equation}\label{eq:ellrlambda}
\intf{\ell, r_{\lambda}}_L=
-1-\frac{\intf{\lambda, \lambda}_{\Leech}\spm}{2}-\frac{\intf{v, v}_{\Leech}\spm}{2}+\intf{v, \lambda}_{\Leech}\spm=
-1+\frac{d(v, \lambda)^2}{2}.
\end{equation}
Therefore $\ell$ belongs to $\barCCC (w_{\Leech})$ if and only if $v$ is a deep hole $c$, and 
in this case, we have 
 $\ell=\bar{f}(c)$.
\end{proof}
Let $c\in \Leech\spm\tensor\Q$ be a deep hole.
We have $\intf{f(c), r_{\lambda}}_L\in \Z_{\ge 0}$
for any $\lambda\in \Leech\spm$.
For $\nu \in \Z_{\ge 0}$, we put
\[
\XXi_{\nu} (c):=
\set{r_{\lambda}\in \LLL(w_{\Leech})}{\intfL{r_{\lambda}, f(c)}=\nu}.
\]
By~\eqref{eq:ellrlambda}, 
we see that $\XXi_{\nu} (c)$ is in one-to-one correspondence with the set
\[
P_{\nu}(c):=\set{\lambda\in \Leech\spm}{d(c, \lambda)^2=2(1+\nu/h)}
\]
by the bijection $\lambda\mapsto r_{\lambda}$ between $\Leech\spm$ and $\LLL(w_{\Leech})$.
Note that these definitions are compatible with the definitions of $\Xi_0(c)$ and $P_0(c)$ in 
Section~\ref{sec:Deepholes}.
The set $P_0(c)$ is the set of points of $\Leech$ nearest to $c$,
and $P_1(c)$ is the set of points of  $\Leech$ next nearest to $c$.
\begin{remark}\label{rem:affinesubspace}
The intersection  form $\intf{\phantom{a}, \phantom{a}}_L$ of $L\tensor\R$ restricted to the affine subspace 
of $L_{\Leech}\tensor\R$ defined by $\intfL{w_{\Leech}, x}=1$ and $\intfL{f(c), x}=\nu$
is an inhomogeneous quadratic form whose homogeneous part of degree $2$ is negative-definite.
Hence we can  explicitly calculate the set $\XXi_{\nu}(c)$.
Then we  obtain  the set $P_{\nu}(c)$.
\end{remark}
We investigate the sets $\XXi_{0}$ and $\XXi_{1}$.
Propositions~\ref{prop:msum},~\ref{prop:thetais}, and~\ref{prop:XXi1M} below
were observed in~\cite{CS1982Note}.
We can also confirm them by looking at the computational data in~\cite{compdataLN}.
As will be explained in Section~\ref{subsec:virtual},
they have geometric meanings in terms of the virtual $K3$ surface $\XX$.
\par
Recall that $\XXi_0 (c)$ forms an extended $\ADE$-configuration of type $\tau(c)$.
We write $\tau(c)$ as 
 \[
  \tau(c)=\tau(c)_1+\cdots+\tau(c)_K,
  \]
  where $\tau(c)_i$ are the $\ADE$-types of 
  the connected components 
  of the dual graph of $\XXi_0 (c)$.
 Let
 \[
\XXi_0 (c)=\XXi_0 (c)_1\sqcup\dots\sqcup \XXi_0 (c)_K
 \]
 be the corresponding decomposition. 
Then we have  a function $ m\colon  \XXi_0 (c)_i\to \Z_{>0}$
defined in Section~\ref{subsec:m} for $i=1, \dots, K$.
\begin{proposition}\label{prop:msum}
 We have
% \[
 $\displaystyle{\sum_{r\in \XXi_0 (c)_i} \; m(r) r=f(c)}$
% \]
for $i=1, \dots, K$.
 \qed
\end{proposition}
Next, we investigate the set $\XXi_1(c)$.
\begin{proposition}\label{prop:thetais}
Let $s$ be an element of $\XXi_1(c)$.
Then, for each $i=1, \dots, K$,  
there exists a unique element $\theta(i, s)$ of $\XXi_0(c)_i$
such that,  for all $r\in \XXi_0(c)_i$, we have 
\[
\intf{r, s}_L=\begin{cases}
1  & \textrm{if}\;\;r=\theta(i, s), \\
0  & \textrm{otherwise.}
\end{cases}
\]
We then have $m(\theta(i, s))=1$, and hence
\begin{equation}\label{eq:Thetacsi}
\Theta (c, s)_i:=\XXi_0(c)_i\setminus \{\theta(i, s)\}
\end{equation}
forms an ordinary $\ADE$-configuration of type $\tau(c)_i$.
\qed
\end{proposition}
We choose and fix an element 
\begin{equation}\label{eq:choosez}
z\in \XXi_1(c),
\end{equation}
and let $U(c, z)$  denote the hyperbolic plane in $L_{\Leech}$ generated by $f(c)$ and $z$.
Its orthogonal complement $U(c, z)\sperp$ in $L_{\Leech}$
contains 
the set
\[
\Theta (c, z):=\Theta (c, z)_1 \sqcup\cdots \sqcup \Theta (c, z)_K
\]
of $(-2)$-vectors  that form an ordinary $\ADE$-configuration of type  $\tau(c)$,
where  $\Theta (c, z)_i$ is defined  by~\eqref{eq:Thetacsi}. %in Proposition~\ref{prop:thetais}.
Let $N$ be the Niemeier lattice such that $\tau(c)=\tauN$.
Since $U(c, z)\sperp$ is unimodular and of rank $24$,
the lattice
\[
N\spm (c, z):=U(c, z)\sperp
\]
is isomorphic to $N\spm$, and contains $\Theta (c, z)$ as a simple root system.
Thus $L_{\Leech}$ has two orthogonal direct-sum decompositions 
\begin{equation}\label{eq:LLeechUN}
L_{\Leech}=U_{\Leech}\oplus \Leech\spm = U(c, z)\oplus N\spm (c, z).
\end{equation}
Let $\gen{\Theta (c, z)}$ be the sublattice of $N\spm(c, z)$ generated by $\Theta (c, z)$, and 
we put
\[
\code (c, z):=N\spm(c, z)/\gen{\Theta (c, z)},
\]
which is a finite abelian group isomorphic to the code of $N\spm$.
We have a natural homomorphism 
\begin{equation}\label{eq:LtoM}
L_{\Leech}\;\to\; N\spm(c, z) \;\to\; \code (c, z),
\end{equation}
where $L_{\Leech}\to N\spm(c, z)$ is the  projection by the second decomposition~\eqref{eq:LLeechUN}.
\begin{proposition}\label{prop:XXi1M}
The mapping~\eqref{eq:LtoM} induces a bijection $\XXi_1(c)\cong  \code (c, z)$.
\qed
\end{proposition}
\begin{remark}\label{rem:thetaiz}
The $(-2)$-vector $\theta(i, z)\in \XXi_0(c)_i$ satisfies the following:
\begin{enumerate}[(a)]
\item $\intf{f(c), \theta(i, z)}_L=0$, $\intf{z, \theta(i, z)}_L=1$, 
\item  if $j\ne i$, then $\intf{r, \theta(i, z)}_L=0$ for all $r\in \Theta(c, z)_j$,
and 
\item $\Theta(c, z)_i\cup\{\theta(i, z)\}$ forms an extended $\ADE$-configuration of type $\tau(c)_i$.
\end{enumerate}
Since $f(c), z$ and the $24$ vectors in $\Theta(c, z)$ span $L_{\Leech}\tensor \Q$,
these properties characterize the vector $\theta(i, z)\in \XXi_0(c)_i$
uniquely.
\end{remark}
%
%
%%%%%%%%%%%%%%%%%%%%%%%%%%%%%%%%
%
%
\section{The virtual $K3$ surface $\XX$}\label{sec:XX}
The results in Section~\ref{sec:shape}
can be interpreted as geometric results 
on a virtual, non-existing $K3$ surface $\XX$.
\subsection{$K3$ surfaces}\label{subsec:K3}
First, we give a brief review of lattice theoretic aspects of the theory of (non-virtual) $K3$ surfaces.
See the book~\cite[Chapter 11]{MWLbook} for details.
See also~\cite{Shimada2022} for a review  
from a computational point of view.
\par
For simplicity, we work over an algebraically closed field of characteristic $\ne 2,3$.
Let $X$ be a $K3$ surface.
We denote by  $S_X$ the N\'eron-Severi lattice of $X$,
that is, $S_X$ is the lattice of numerical equivalence classes of divisors on $X$.
Let $\intf{\nullvars}_S$ denote  
the intersection form of $S_X$.
For a curve $C$ on $X$, let $[C]\in S_X$ be  the class of $C$.
Suppose that the Picard number $\rank S_X$ is $>1$.
Then $S_X$ is an even hyperbolic lattice.
Let $\PPP_X$ be  the positive cone of $S_X$ that contains an ample class.
The \emph{nef-and-big cone} $\NNN_X$ of $X$ is defined by 
\[
\NNN_X:=\set{x\in \PPP_X}{\intf{x, [C]}_S\ge 0\;\;\textrm{for all curves $C$ on $X$}}.
\]
%of $X$ is the standard fundamental domain of the Weyl group $W(S_X)$ containing an ample class.
We denote by $\barNNN_X$ the closure of $\NNN_X$ in $\barPPP_X$.
We  put 
\[
\Rats(X):=\set{r\in S_X}{\textrm{$r$ is the class of a smooth rational curve on $X$}}.
\]
Then we have the following: % well-known result.
\begin{proposition}\label{prop:Rats}
\begin{enumerate}[{\rm (1)}]
\item 
The nef-and-big cone $\NNN_X$ is 
the standard fundamental domain of the Weyl group $W(S_X)$ containing an ample class.
\item A $(-2)$-vector $r$ of $S_X$ belongs to $\Rats(X)$
if and only if $r$ defines a wall of the standard fundamental domain $\NNN_X$.
\qed
\end{enumerate}
\end{proposition}
Let $f$ and $z$ be  vectors of $S_X$  such that  
\[
\intf{f, f}_S=0, 
\quad
\intf{f,z}_S=1,
\quad
\intf{z,z}_S=-2, 
\quad
f\in \barNNN_X,
\quad 
z\in \Rats(X).
\]
Then we have an elliptic fibration $\phi\colon X\to \P^1$
with a section $\zerosec\colon \P^1\to X$
such that $f$ is the class of a fiber of $\phi$ and that $z$ is the class of the image of $\zerosec$.
Let $U_{f,z}$ denote the hyperbolic plane in  $S_X$ generated by $f$ and $z$, and 
let  $U_{f,z}\sperp$ be the orthogonal complement of $U_{f,z}$ in $S_X$.
Note that $U_{f,z}\sperp$ is negative-definite.
Let $R(U_{f,z}\sperp)$ be the set of $(-2)$-vectors in $U_{f,z}\sperp$,
and $\gen{R(U_{f,z}\sperp)}$ the sublattice of $U_{f,z}\sperp$ generated by $R(U_{f,z}\sperp)$.
We put 
\[
\tilTheta_{\phi}:=\set{[C]}{\textrm{$C$ is a smooth rational curve on $X$ mapped to a point by $\phi$}}.
\]
Then  $\tilTheta_{\phi}$ forms an extended $\ADE$-configuration.
Let 
\[
\tilTheta_{\phi}=\tilTheta_{\phi, 1}\sqcup\dots\sqcup \tilTheta_{\phi, K}
\]
be the decomposition according to the connected components of the dual graph of $\tilTheta_{\phi}$.
Then the fibration $\phi\colon X\to \P^1$ has exactly $K$ reducible fibers $\phi\inv (p_1), \dots, \phi\inv (p_K)$,
and, under an appropriate  numbering of the points $p_1, \dots, p_K\in \P^1$, 
we have
\[
[C]\in  \tilTheta_{\phi, i}\;\;\Longleftrightarrow\;\; C\subset \phi\inv (p_i)
\]
for a smooth rational curve $C$ on $X$.
Let $m\colon \tilTheta_{\phi, i}\to \Z_{>0}$ be the function defined in Section~\ref{subsec:m}.
Then we have 
\begin{equation}\label{eq:phistar}
\phi^* (p_i)=\sum_{C\subset \phi\inv (p_i)} m([C]) C,
\end{equation}
where $C$ runs through the set of irreducible components of $\phi\inv (p_i)$.
Let $C_{i0}$ be the unique irreducible component of $\phi\inv (p_i)$ that intersects 
the zero section $\zeta$, and 
we put
\[
\Theta_{\phi, \zeta, i}:=\tilTheta_{\phi, i}\setminus\{[C_{i0}]\}.
\]
Then  each $\Theta_{\phi, \zeta, i}$ forms  a connected  ordinary $\ADE$-configuration. % of type $\tau_{\phi, i}$.
We put 
\[
\Theta_{\phi, \zeta}:=\Theta_{\phi, \zeta, 1}\sqcup\dots\sqcup \Theta_{\phi, \zeta, K}.
\]
%
%which is 
% the set of classes of smooth rational curves in fibers of $\phi$ that are disjoint from $\zeta$.
 Then we have
 \[
 \Theta_{\phi, \zeta}=R(U_{f,z}\sperp)\cap \Rats(X).
 \]
We can  calculate the classes 
\[
[C_{i0}]=f-\sum_{r \in\Theta_{\phi, \zeta, i}} m(r)r
\]
of smooth rational curves in fibers of $\phi$ that intersect $\zeta$.
Note that $\sum_{r \in\Theta_{\phi, \zeta, i}} m(r)r$ is the highest root of $\gen{\Theta_{\phi, \zeta, i}}$ with respect to $\Theta_{\phi, \zeta, i}$.
\par
Let $\MW_{\phi, \zeta}$ be the Mordell-Weil group of the Jacobian fibration $\phi\colon X\to \P^1$
with the zero section $\zeta\colon \P^1\to X$, 
that is, $\MW_{\phi, \zeta}$ is the group of sections $s\colon \P^1\to X$ 
of $\phi$ with $\zeta$ being $0$.
For $s\in \MW_{\phi, \zeta}$, let $[s]\in \Rats(X)$ denote the class of the image of the section $s$.
Then we have the following famous result. See~\cite{MWLbook}.
\begin{proposition}\label{prop:MW}
The mapping $s\mapsto [s]$ induces an isomorphism
\begin{equation}\label{eq:MW}
 \MW_{\phi, \zeta}\;\;\cong\;\; S_X/(U_{f,z}\oplus \gen{\Theta_{\phi, \zeta}})
\end{equation}
of abelian groups.
\qed
\end{proposition}
\begin{remark}\label{rem:Shimada2022}
We have an algorithm~\cite[Section 4]{Shimada2022} that calculates,
for any $v\in S_X$,
the class $[s_v]\in \Rats(X)$ of the section $s_v\in  \MW_{\phi, \zeta}$
corresponding to the class of $v$ modulo $ U_{f,z}\oplus \gen{\Theta_{\phi, \zeta}}$ via~\eqref{eq:MW}.
\end{remark}
\subsection{Geometry  of the virtual $K3$ surface $\XX$}\label{subsec:virtual}
Let $\XX$ be a \emph{virtual} $K3$ surface with an isometry
\begin{equation}\label{eq:SLisom}
S_{\XX}\cong L_{\Leech}.
\end{equation}
Applying  to $\XX$ the results of  $K3$ surfaces
explained in Section~\ref{subsec:K3},
we obtain natural explanations to the  results 
Propositions~\ref{prop:msum},~\ref{prop:thetais}, and~\ref{prop:XXi1M}
that were  observed in~\cite{CS1982Note}.
Remark that such a $K3$ surface  $\XX$ does \emph{not} exist,
and the arguments below should be considered only as a heuristic guide.
\par
By composing~\eqref{eq:SLisom} with an element of the Weyl group $W(L_{\Leech})$,
we can assume that~\eqref{eq:SLisom} maps the nef-and-big cone $\NNN_{\XX}$ 
of $\XX$ to the Conway chamber $\CCC(w_{\Leech})$.
In the following, we identify $L_{\Leech}$ with $S_{\XX}$, and $\CCC(w_{\Leech})$ with $\NNN_{\XX}$. %,  by~\eqref{eq:SLisom}.
\par
By Theorem~\ref{thm:ConwayChamber} and Proposition~\ref{prop:Rats},  the set $\Rats(\XX)$ of the classes of smooth rational curves on $\XX$ is equal to 
the set $\LLL(w_{\Leech})$ of Leech roots of $w_{\Leech}$.
The primitive vector $w_{\Leech} \in S_{\XX}\cap \barNNN_{\XX}$ 
corresponds to the class of a fiber of an elliptic fibration 
\[
\phi_{\Leech} \colon \XX\to \P^1.
\]
Since $\intf{w_{\Leech}, r_{\lambda}}_L=1$ for any $r_{\lambda}\in \LLL(w_{\Leech})$,
every smooth rational  curve on $\XX$ is a section of $\phi_{\Leech} $.
If we choose a smooth rational  curve and consider it as a zero section $\zeta\colon \P^1\to\XX$ of  $\phi_{\Leech}$, then, 
by Proposition~\ref{prop:MW},
the Mordell-Weil group of this Jacobian fibration $(\phi_{\Leech}, \zeta)$  is isomorphic to $\Z^{24}$,
because we have $\Theta_{\phi_{\Leech}, \zeta}=\emptyset$.
\begin{remark}
More strongly, 
the Mordell-Weil lattice~(see~\cite{MWLbook}) of $(\phi_{\Leech}, \zeta)$ is isomorphic to the Leech lattice $\Leech$.
\end{remark}
Let $c\in \Leech\spm\tensor\Q$ be a deep hole.
We have defined in~\eqref{eq:definitionfc}
a  primitive vector $f(c)\in L_{\Leech}$ 
generating a half-line in $\barCCC(w_{\Leech})\cap \bdr \barPPP_{L}$.
Then the  vector $f(c) \in S_{\XX}\cap \barNNN_{\XX}$ 
is the class of a fiber of an elliptic fibration 
\[
\phi(c) \colon \XX\to \P^1.
\]
%Let $N$ be the corresponding Niemeier lattice.
The set $\XXi_0(c)$ is the set of classes of smooth rational curves on $\XX$ that are contained in fibers of  $\phi(c) $.
Let $\phi(c)^*(p_i)$  ($i=1, \dots, K$) be the reducible fibers of $\phi(c)$.
Renumbering the points $p_1, \dots, p_K\in \P^1$ if necessary,  we have that, 
for each $i=1, \dots, K$, the set $\XXi_0(c)_i$ is the set of classes 
$[C]$ of irreducible components $C$  of  $\phi(c)^*(p_i)$
with $m([C])\in \Z_{>0}$ being the multiplicity of $C$ in the fiber $\phi(c)^*(p_i)$.
The set $\XXi_1(c)$ is the set of classes of sections of $\phi(c) $. 
Hence Proposition~\ref{prop:msum} follows from~\eqref{eq:phistar},
Proposition~\ref{prop:thetais} follows from the fact that a section and a fiber intersect only at one point and with intersection multiplicity $1$.
We choose $z\in \XXi_1(c)$ as in~\eqref{eq:choosez}.
Then we have a section $\zeta\colon \P^1\to \XX$ whose class is $z$.
We consider $\zeta$ as a zero section of $\phi(c)$.
Then $\theta(i, z)$ is the class of the smooth rational curve in $\phi(c)^*(p_i)$ that intersects  $\zeta$, 
and hence $\Theta(c, z)$ is the set of classes of  irreducible components of  reducible fibers of  $\phi(c)$ 
that are disjoint from $\zeta$.
Hence 
the Mordell-Weil group of the Jacobian fibration $(\phi(c), \zeta)$ is isomorphic to $\code (c, z)=N\spm(c, z)/\gen{\Theta (c, z)}$
by Proposition~\ref{prop:MW}, and 
Proposition~\ref{prop:XXi1M} also follows from Proposition~\ref{prop:MW}.
%
%
%%%%%%%%%%%%%%%%%%%%%%%%%%%%%
%
%
\section{Constructions of the Leech lattice}\label{sec:costruction}
Let $N$ be a  Niemeier lattice with roots.
Let
\[
\Theta=\Theta_1\sqcup\dots\sqcup \Theta_K\;\;\subset\;\; R \;\;\subset\;\; \gen{R}  \;\;\subset\;\; N\spm
\]
and $\tauN=\tauN_1+\cdots +\tauN_K$ be defined as in Section~\ref{sec:Niemeier}.
We give a construction of the Leech lattice $\Leech\spm$ for each codeword $\gamma$ of the finite abelian group $N\spm/\gen{R}$.
\begin{remark}
By Proposition~\ref{prop:transitive}, 
this construction does not depend on the choice of the simple root system $\Theta$ of $\gen{\Theta}=\gen{R}$
up to the action of $\OG(N\spm)$.
\end{remark}
First we present a lemma about the discriminant group of
a negative-definite root lattice.  % whose  $\ADE$-type is the $\ADE$-type of a connected $\ADE$-configuration.
Let $\Sigma=\{r_1, \dots, r_k\}$ be a set of $(-2)$-vectors that form a \emph{connected} ordinary 
$\ADE$-configuration, and let $\gen{\Sigma}$ denote  the lattice generated by $\Sigma$.
Let $m\colon \Sigma\to \Z_{>0}$ be the function defined in Section~\ref{subsec:m}.
We put 
\[
J(\Sigma):=\set{j}{m(r_j)=1}\;\;\subset\;\;\{1, \dots, k\}.
\]
Let $r_1\dual, \dots, r_k\dual$ be the basis of $\gen{\Sigma}\dual$ dual to 
the basis $r_1, \dots, r_k$ of $\gen{\Sigma}$.
Then the following can be checked easily for each connected $\ADE$-configuration $\Sigma$. % of type $A_l$, $D_m$, and $E_n$.
\begin{lemma}\label{lem:Jbij}
The mapping $j\mapsto r_j\dual \bmod \gen{\Sigma}$ gives rise to  a bijection from 
$J(\Sigma)$ to 
the set of non-zero elements of the discriminant group $\gen{\Sigma}\dual / \gen{\Sigma}$.
%$\gen{\Sigma}\dual / \gen{\Sigma} \setminus\{0\}$.
\qed 
\end{lemma}
\begin{definition}\label{def:canonical0}
For a codeword $\alpha\in \gen{\Sigma}\dual / \gen{\Sigma} $,
we define its \emph{canonical representative} $\tilde{\alpha}\in \gen{\Sigma}\dual$
by the following:
if $\alpha=0$, then $\tilde{\alpha}=0$,
and 
if $\alpha\ne 0$, then $\tilde{\alpha}=r_j\dual$,
%where $j$ is the index in $J(\Sigma)$ that corresponds to $\alpha$ by the bijection in 
%Lemma~\ref{lem:Jbij}.
where $j\in J(\Sigma)$ corresponds to $\alpha$ by the bijection in 
Lemma~\ref{lem:Jbij}.
\end{definition}
We put
\[
A_i:=\gen{\Theta_{i}}\dual/\gen{\Theta_{i}}.
\]
By the natural embedding 
\begin{equation}\label{eq:dualembedding}
N\spm\inj \gen{R}\dual=\gen{\Theta_1}\dual\oplus \cdots \oplus \gen{\Theta_K}\dual, 
\end{equation}
we have an embedding 
\begin{equation}\label{eq:codeembedding}
N\spm/\gen{R}\;\;\inj\;\; A_1\times \cdots\times A_K.
\end{equation}
Let $\gamma\in N\spm/\gen{R}$ be a codeword.
For $i=1, \dots, K$,
let $\gamma_i\in A_i$ be the $i$th component of $\gamma$ by the embedding~\eqref{eq:codeembedding}.
Let $\tilde{\gamma}_i\in \gen{\Theta_i}\dual$ be the canonical representative of $\gamma_i$, and 
we put
\[
v_{\gamma}:=\tilde{\gamma}_1+ \cdots +\tilde{\gamma}_K \;\;\in\;\;  \gen{\Theta_1}\dual\oplus\cdots\oplus  \gen{\Theta_K}\dual=\gen{R}\dual.
\]
Then we have the following:
\begin{proposition}
We have $v_{\gamma}\in N\spm$.
\end{proposition}
\begin{proof}
Let $U_N$ be a copy of the hyperbolic plane $U$, and we put
\[
L_N:=U_N\oplus N\spm.
\]
The intersection form on $L_N$ is denoted by $\intf{\phantom{a}, \phantom{a}}_L$.
Let $u_0, u_1$ be the basis of $U_N$ given in Example~\ref{example:U}.
A vector of $L_N$ is written as 
\begin{equation}\label{eq:LNvect}
(a, b, v)_N=a u_0 + b u_1 +v,
\quad\textrm{where $a, b\in \Z$ and $v\in \Nm$.}
\end{equation}
By $v\mapsto (0,0,v)_N$, 
we regard $N\spm$ as a sublattice of $L_N$.
In particular, we have $R\subset L_N$ and $\Theta\subset L_N$.
We then put
\[
f_N:=(1,0,0)_N \in L_N, \quad 
z_N:=(-1, 1, 0)_N \in L_N.
\]
Note that $L_N$ is isomorphic to $L_{\Leech}$.
We construct an  isometry $L_{\Leech}\cong L_N$  explicitly by means of deep holes.
Let $c\in \Leech\spm \tensor\Q$ be a deep hole such that $\tau(c)=\tauN$.
Recall from Section~\ref{sec:shape} that 
we have defined subsets $\XXi_0(c)$ and $\XXi_1(c)$ of the set $\LLL(w_{\Leech})$ of the Leech roots of $w_{\Leech}$,
and for a fixed element $z\in \XXi_1(c)$,  we constructed an orthogonal direct-sum decomposition 
\begin{equation*} \label{eq:LLeechUN2}
L_{\Leech}=U(c, z)\oplus N\spm (c, z).
\end{equation*}
Note that $N\spm(c, z)$ is isomorphic to $N\spm$.
By Proposition~\ref{prop:transitive}, after renumbering the connected components $\Theta_1, \dots, \Theta_K$ of 
the simple root system $\Theta$ of $N\spm$  if necessary,
we have an isometry $N\spm (c, z)\cong N\spm$ that maps $\Theta(c, z)_i$ to $\Theta_i$ for $i=1, \dots, K$.
Then we obtain an isometry 
\begin{equation}\label{eq:LLeechLN}
L_{\Leech}\cong L_N
\;\;
\textrm{satisfying}
\;\;
f(c)\mapsto f_N,
\;\;
z\mapsto z_N, 
\;\;
\Theta(c, z)_i\isom \Theta_i\;(i=1, \dots, K).
\end{equation}
For $i=1, \dots, K$, 
let $\mu_i\in \gen{\Theta_i}$ be the highest root with respect to $\Theta_i$.
We put
\[
\theta_i:=(1, 0, -\mu_i)_N.
\]
The $(-2)$-vector $\theta_i\in L_N$ satisfies the following:
\begin{enumerate}[(a)]
\item
$\intf{f_N, \theta_i}_L=0$, $\intf{z_N, \theta_i}_L=1$, 
\item
if $j\ne i$, then $\intf{r, \theta_i}_L=0$ for all $r\in \Theta_j$,
and 
\item $\Theta_i\cup\{\theta_i\}$ form an extended $\ADE$-configuration of type $\tau_i$.
\end{enumerate}
By Remark~\ref{rem:thetaiz},
we see that  $\theta_i\in L_N$ corresponds to $\theta(i, z)\in \XXi_0(c)_i\subset L_{\Leech}$ via the isometry~\eqref{eq:LLeechLN}
given above.
We put
\[
\widetilde{\Theta}_i:=\Theta_i\cup\{\theta_i\}.
\]
Then the isometry~\eqref{eq:LLeechLN} induces a bijection from $\XXi_0(c)_i$ to $\widetilde{\Theta}_i$.
\par
The isometry~\eqref{eq:LLeechLN} induces an isomorphism
\begin{equation}\label{eq:MN}
\code (c, z)=N\spm (c, z)/\gen{\Theta(c,z)} \cong N\spm/\gen{R}.
\end{equation}
Let $s_{\gamma}\sprime\in \XXi_1(c)$ denote the $(-2)$-vector of $L_{\Leech}$ that corresponds to 
the codeword $\gamma \in N\spm/\gen{R}$ via the isomorphism~\eqref{eq:MN} and 
the bijection $ \XXi_1(c) \cong \code (c, z)$ in Proposition~\ref{prop:XXi1M}.
Let $s_{\gamma}$ be the $(-2)$-vector of $L_{N}$ that corresponds to 
the $(-2)$-vector $s_{\gamma}\sprime \in L_{\Leech}$ via the isomorphism~\eqref{eq:LLeechLN}.
Since $\intf{f_N, s_{\gamma}}_L=\intf{f(c), s_{\gamma}\sprime}_L=1$ and $\intf{s_\gamma, s_\gamma}_L=-2$, 
there exists a vector $u_{\gamma}\in N\spm$
such that 
 \[
 s_{\gamma}=(a, 1, u_{\gamma})_N, \quad\textrm{where $a=-1-\intf{ u_{\gamma},  u_{\gamma}}_N\spm/2$}.
 \]
By Proposition~\ref{prop:thetais} transplanted to $ s_{\gamma} \in L_{N}$ from $ s_{\gamma}\sprime \in L_{\Leech}$, 
we see that the $i$th component of $u_{\gamma}$
by the embedding~\eqref{eq:dualembedding}
 is equal to the canonical representative $\tilde{\gamma}_i$.
 Indeed, Proposition~\ref{prop:thetais} implies that 
 there exists a unique element $r_{\gamma}\in \widetilde{\Theta}_i$ 
 such that $m(r_{\gamma})=1$ and  that, for any $r\in \widetilde{\Theta}_i$, we have 
 \[
 \intf{s_{\gamma}, r}_L=\begin{cases}
 1 & \textrm{if $r=r_{\gamma}$}, \\
 0 & \textrm{otherwise}.
 \end{cases}
 \]
 Hence $s_{\gamma}$ intersects the elements of $\Theta_i$
 with the same intersection numbers as the canonical representative $\tilde{\gamma}_i \in \gen{\Theta_i}\dual$ does.
(If $r_{\gamma}=\theta_i$, then we have $\tilde{\gamma}_i=0$.)
Therefore  we obtain $v_{\gamma}=u_{\gamma}\in N\spm$.
\end{proof}
\begin{definition}\label{def:canonical1}
We call $v_{\gamma}\in N\spm$ the \emph{canonical representative} of $\gamma \in N\spm/\gen{R}$.
\end{definition}
Recall from Section~\ref{sec:Niemeier} that $h$ is the Coxeter number of $N\spm$ 
and that $\rho\in N\spm\tensor\Q$ is the Weyl vector of $N\spm$
with respect to $\Theta$.
\begin{proposition}\label{prop:wN}
The vector 
\[
w_N:=(h+1, h, \rho)_N
\]
corresponds to the Weyl vector $w_{\Leech}\in L_{\Leech}$ via the isometry~\eqref{eq:LLeechLN}.
In particular, the lattice $(\Z  w_N)\sperp/ \Z  w_N$ is isomorphic to $\Leech\spm$.
\end{proposition}
\begin{remark}\label{rem:rho2}
The vector $w_N$ appeared in Borcherds~\cite{Borcherds1985}.
Proposition~\ref{prop:wN} gives a proof that $\rho\in N\spm\tensor\Q$ is in fact $\rho\in N\spm$.
\end{remark}
\begin{proof}
Note that equality~\eqref{eq:rhonorm}  implies $\intf{w_N, w_N}_L=0$.
From defining property  (b) of $\rho_i$   and property (c) of $h=h_i$   in Definition~\ref{def:Coxeter number}, we see that 
the vector  $w_N \in L_{N}\tensor \Q$
satisfies 
$\intf{w_N, z_N}_L=1$  and 
\[
\quad \intf{w_N, r}_L=1 \quad\textrm{for all $r\in \Theta$}, \qquad 
\intf{w_N, \theta_i}_L=1 \quad\textrm{for $i=1, \dots, K$}.
\]
Since $z_N$, $r$ ($r\in \Theta$) and $\theta_i$ ($i=1, \dots, K$) span $L_{N}\tensor \Q$ and correspond,
via the isometry~\eqref{eq:LLeechLN}, 
 to Leech roots of $w_{\Leech}$, 
we see that $w_N$  corresponds to $w_{\Leech}$ 
via the isometry~\eqref{eq:LLeechLN}.
\end{proof}
\begin{theorem}\label{thm:constC}
Let $\gamma$ be a codeword of the code $\Nm/\gen{R}$, 
and let $v_\gamma\in \Nm$ be the canonical representative of $\gamma$.
We put
\[
n_\gamma:=\intfNm{v_\gamma,v_\gamma},
\quad
a_\gamma:=2h +1+ h\,n_\gamma/2.
\]
We define linear forms $\alpha_0\colon \Nm\to \Q$ and $\alpha_1\colon \Nm\to \Q$ by 
\begin{eqnarray*}
\alpha_0(u) &:=&\intfNm{h v_\gamma -\rho, u}/a_{\gamma},  \\
\alpha_1(u) &:=&\left(1+n_{\gamma}/2\right)\alpha_0(u)-\intfNm{v_\gamma, u}, 
\end{eqnarray*}
and put 
\[
\newLeech\spm(\gamma):=\set{u\in \Nm}{\alpha_0(u)\in \Z}.
\]
Then the $\Z$-module $\newLeech\spm (\gamma)$ with the intersection form
\begin{equation}\label{eq:newform}
\intf{u, u\sprime}:=\intfNm{u, u\sprime}+\alpha_0 (u) \alpha_1(u\sprime)+\alpha_1 (u) \alpha_0(u\sprime)
\end{equation}
is isomorphic to the negative-definite Leech lattice $\Leech\spm$.
\end{theorem}
\begin{proof}
Recall that $s_{\gamma}=(-1-n_{\gamma}/2, 1, v_{\gamma})_N \in L_{N}$ is the vector corresponding, via  the isometry~\eqref{eq:LLeechLN}, 
to the Leech root
$s_{\gamma}\sprime \in \XXi_1(c)$ of $w_{\Leech}$.
Let $U(w_N, s_{\gamma})$ be the hyperbolic plane in $L_N$ generated by $w_N$ and $s_{\gamma}$.
Since $w_N$ is a Weyl vector of $L_N$, 
the orthogonal complement  $U(w_N, s_{\gamma})\sperp\cong (\Z w_N)\sperp / \Z w_N$ is isomorphic to the negative-definite Leech lattice $ \Leech\spm$.
A vector $(x, y, u)_N$ of $L_N\tensor \Q$ is orthogonal to both of $w_N$ and $s_\gamma$ 
if and only if
\[
x=\alpha_1(u), \quad y=\alpha_0(u).
\]
Hence the image of the orthogonal projection 
\[
U(w_N, s_{\gamma})\sperp\inj L_N=U_N\oplus N\spm \to N\spm
\]
is equal to the $\Z$-submodule $\newLeech\spm (\gamma)$ of $N\spm$,
and the restriction of $\intf{\phantom{a}, \phantom{a}}_L$ to $U(w_N, s_{\gamma})\sperp$ 
gives rise  to the intersection form~\eqref{eq:newform} on $\newLeech\spm (\gamma)$.
\end{proof}
\begin{remark}
In terms of $\XX$, 
the construction above is described as follows.
The sublattice $U_N\subset L_N$ yields a Jacobian fibration of $\XX$
by the isometry $L_N\cong L_{\Leech}=S_{\XX}$ given in~\eqref{eq:LLeechLN}, and 
the $(-2)$-vector $s_{\gamma}$ is the class of the image of the element 
of the Mordell-Weil group corresponding to $\gamma$.
\end{remark}
Considering the case where $\gamma=0$ in Theorem~\ref{thm:constC}, 
and changing signs of  intersection forms of lattices,
we obtain Corollary~\ref{cor:zerogamma} in Introduction.
\begin{remark}\label{rem:canonicalreps}
Since the Leech lattice $\Leech\spm$ is characterized, up to isomorphism,  as the unique even unimodular negative-definite  lattice of rank $24$
with no roots~(see~\cite{Conway1969Note}), we can confirm Theorem~\ref{thm:constC} by direct computation, 
once  we  compute  canonical representatives $v_{\gamma}$ of codewords $\gamma\in N\spm/\gen{R}$ explicitly.
\par
The canonical representatives are computed as follows.
The set of $(-2)$-vectors $s_{\gamma}\in L_N$, where $\gamma$ runs through $N\spm/\gen{R}$,  is equal to
\[
\set{r\in L_N}{\;\intf{f_N, r}_L=\intf{w_N, r}_L=1, \;\intf{r,r}=-2\;},
\]
and,  as was explained in Remark~\ref{rem:affinesubspace}, 
this set can be computed easily
as the set of  integer solutions of a negative-definite inhomogeneous quadratic form. 
The set of canonical representatives $v_{\gamma}$ is then obtained from this set  by the projection $L_N\to N\spm$. 
\par
We can also use the 
algorithm for the study of elliptic $K3$ surfaces 
described  in~\cite[Section 4]{Shimada2022}. 
See Remark~\ref{rem:Shimada2022}.
\end{remark}
\bibliographystyle{plain}
\bibliography{myrefsAut}
\end{document}